\documentclass[12pt]{article}
\usepackage{amsmath, amsthm, amssymb, amsfonts, mathrsfs}
\usepackage[dvips]{graphicx}
\newtheorem{theorem}{Theorem}[section]
\newtheorem{cor}[theorem]{Corollary}

\newtheorem{lem}[theorem]{Lemma}
\newtheorem{prop}[theorem]{Proposition}

\def \Zl {{\mathbb Z}}
\def \Nl {{\mathbb N}}
\def \Rl {{\mathbb R}}

\def \Cl {{\mathbb C}}

\def \vl {{\mathbf v}}

\title{More Circulant Graphs exhibiting Pretty Good State Transfer}

\author{ Hiranmoy Pal\\
Department of Mathematics\\
Indian Institute of Technology Guwahati\\Guwahati, India - 781039\\
Email: hiranmoy@iitg.ernet.in
}
\begin{document}
\maketitle

\vspace{-0.3in}

\begin{center}{Abstract}\end{center}
The transition matrix of a graph $G$ corresponding to the adjacency matrix $A$ is defined by $H(t):=\exp{\left(-itA\right)},$ where $t\in\mathbb{R}$. The graph is said to exhibit pretty good state transfer between a pair of vertices $u$ and $v$ if there exists a sequence $\left\lbrace t_k\right\rbrace$ of real numbers such that $\lim\limits_{k\rightarrow\infty} H(t_k) {\bf e}_u=\gamma {\bf e}_v$, where $\gamma$ is a complex number of unit modulus. We classify some circulant graphs exhibiting or not exhibiting pretty good state transfer. This generalize several pre-existing results on circulant graphs admitting pretty good state transfer.
\newpage
\section{Introduction}
In \cite{god1}, Godsil first introduced the problem of pretty good state transfer (PGST) in the conext of continuous-time quantum walks. State transfer has significant applications in quantum information processing and cryptography (see \cite{ben,ek}). We consider state transfer with respect to the adjacency matrix of a graph. It can be considered with respect the Laplacian matrix as well. However, for regular graphs both considerations provide similar results. All graphs, we consider, are simple, undirected and finite. Now we define PGST on a graph $G$ with adjacency matrix $A$. The transition matrix of $G$ is defined by
\[H(t):=\exp{\left(-itA\right)},~\text{where}~t\in\Rl.\]
Let ${\bf e}_u$ denote the characteristic vector corresponding to the vertex $u$ of $G$. The graph $G$ is said to exhibit PGST between a pair of vertices $u$ and $v$ if there is a sequence $\left\lbrace t_k\right\rbrace$ of real numbers such that
\[\lim_{k\rightarrow\infty} H(t_k) {\bf e}_u=\gamma {\bf e}_v,~\text{where}~\gamma\in\Cl~\text{and}~|\gamma|=1.\]
In the past decade the study of state transfer has received considerable attention. Now we know a handful of graphs exhibiting PGST. The main goal here is to find new graphs that admits PGST. However it is more desirable to find PGST in graphs with large diameters. In \cite{god4}, Godsil \emph{et al.} established that the path on $n$ vertices, denoted by $P_n$, exhibits PGST if and only if $n+1$ equals to either $2^m$ or $p$ or $2p$, where $p$ is an odd prime. A complete characterization of cycles exhibiting PGST appears in \cite{pal4}. There we see that a cycle on $n$ vertices exhibits PGST if and only if $n$ is a power of $2$. In \cite{fan}, Fan \emph{et al.} showed that a double star $S_{k,k}$ admits PGST if and only if $4k+1$ is not a perfect square. A NEPS is a graph product which generalizes several well known graph products, namely, Cartesian product, Kronecker product etc. In \cite{pal3}, we see some NEPS of the path on three vertices having PGST. Also, in \cite{ack,ack1}, PGST has been studied on corona products. Some other relevant results regarding state transfer can be found in \cite{god1, pal2, pal1,pal,saxe}. In the present article, we find some circulant graphs exhibiting or not exhibiting PGST. The results established in this article generalize several pre-existing results appearing in \cite{pal4} on circulant graphs admitting PGST.\par
Let $\left(\Gamma,+\right)$ be a finite abelian group and consider $S\setminus\left\lbrace0\right\rbrace\subseteq\Gamma$ with $\left\lbrace -s:s\in S\right\rbrace=S$. A Cayley graph over $\Gamma$ with the symmetric set $S$ is denoted by $Cay\left(\Gamma,S\right)$. The graph has the vertex set $\Gamma$ where two vertices $a,b\in\Gamma$ are adjacent if and only if $a-b\in S$. The set $S$ is called the connection set of $Cay\left(\Gamma,S\right)$. Let $\Zl_n$ be the cyclic group of order $n$. A circulant graph is a Cayley graph over $\Zl_n$. The cycle $C_n$, in particular, is the circulant graph over $\Zl_n$ with the connection set $\left\lbrace1,n-1\right\rbrace$. The eigenvalues and its corresponding eigenvectors of a cycle are very well known. The eigenvalues of $C_n$ are
\begin{eqnarray}\label{e}
\lambda_l=2\cos{\left(\frac{2l\pi}{n}\right)},\;0\leq l\leq n-1,
\end{eqnarray}
and the corresponding eigenvectors are $\vl_l=\left[1,\omega_n^l,\ldots,\omega_n^{l(n-1)}\right]^T$, where $\omega_n=\exp{\left(\frac{2\pi i}{n}\right)}$ is the primitive $n$-th root of unity.\par
In general, the eigenvalues and eigenvectors of a Cayley graph over an abelian group are also known in terms of characters of the abelian group. In \cite{wal1}, it appears that the eigenvectors of a Cayley graph over an abelian group are independent of the connection set. Therefore the set of eigenvectors of two Cayley defined over an abelian group can be chosen to be equal. Hence we have the following result.
\begin{prop}\label{3ab}
If $S_1$ and $S_2$ are symmetric subsets of an abelian group $\Gamma$ then adjacency matrices of the Cayley garphs $Cay(\Gamma, S_1 )$ and $Cay(\Gamma, S_2 )$ commute.
\end{prop}
Let $G$ and $H$ be two simple graphs on the same vertex set $V$, and the respective edge sets $E(G)$ and $E(H)$. The edge union of $G$ and $H$, denoted $G\cup H$, is defined on the vertex set $V$ whose edge set is $E(G)\cup E(H)$. Using Proposition \ref{3ab} and the exponential nature the transition matrix, we obtain the following result which allows us to find the transition matrix of an edge union of two edge disjoint Cayley graphs.
\begin{prop}\cite{pal1}\label{4ab}
Let $\Gamma$ be a finite abelian group and consider two disjoint and symmetric subsets $S,T\subset\Gamma$. Suppose the transition matrices of $Cay(\Gamma, S)$ and $ Cay(\Gamma, T)$ are $H_{S}(t)$ and $H_{T}(t)$, respectively. Then $Cay(\Gamma, S\cup T)$ has the transition matrix $H_{S}(t)H_{T}(t).$
\end{prop}
Let $n\in\Nl$ and $d$ be a proper divisor of $n$. We denote
\[S_n(d)=\left\lbrace x\in\Zl_n: gcd(x,n)=d\right\rbrace,\]
and for any set $D$ of proper divisors of $n$, we define
\[S_n(D)=\bigcup\limits_{d\in D} S_n(d).\]
The set $S_n(D)$ is called a gcd-set of $\Zl_n$. A graph is called integral if all its eigenvalues are integers. The following theorem determines which circulant graphs are integral.
\begin{theorem}\cite{so}\label{so}
A circulant graph $Cay\left(\Zl_n,S\right)$ is integral if and only if $S$ is a gcd-set.
\end{theorem}
It is therefore clear that if a circulant graph $Cay\left(\Zl_n,S\right)$ is non-integral then there exists a proper divisor $d$ of $n$ such that $S\cap S_n(d)$ is a non-empty proper subset of $S_n(d)$. 

\section{Pretty Good State Transfer on Circulant Graphs}
Let us recall some important observations on circulant graphs exhibiting PGST (see \cite{pal4}). We include the following discussion for convenience. Suppose $G$ is a graph with adjacency matrix $A$. It is well known that the matrix $P$ of an automorphism of $G$ commutes with $A$. By the spectral decomposition of the transition matrix $H(t)$ of $G$, we conclude that $H(t)$ is a polynomial in $A$. Therefore $P$ commutes with $H(t)$ as well. Suppose $G$ admits PGST between two vertices $u$ and $v$. Then there exists a sequence $\left\lbrace t_k\right\rbrace$ of real numbers and $\gamma\in\Cl$ with $|\gamma|=1$ such that
\[\lim_{k\rightarrow\infty} H(t_k) {\bf e}_u=\gamma {\bf e}_v.\]
Further this implies that
\[\lim_{k\rightarrow\infty} H(t_k) \left(P{\bf e}_u\right)=\gamma \left(P{\bf e}_v\right).\]
The sequence $\left\lbrace H(t_k){\bf e}_u\right\rbrace$ cannot have two different limits, and therefore if $P$ fixes ${\bf e}_u$ then $P$ must also fix ${\bf e}_v$. For the circulant graph $Cay\left(\Zl_n,S\right)$, we know that the map sending $j$ to $-j$ is an automorphism. As a consequence, we have the following result.
\begin{lem}\label{l0}\cite{pal4}
If pretty good state transfer occurs in a circulant graph $Cay\left(\Zl_n,S\right)$ then $n$ is even and it occurs only between the pair of vertices $u$ and $u+\frac{n}{2}$, where $u,u+\frac{n}{2}\in\Zl_n$.
\end{lem}
A graph $G$ is called vertex transitive if for any pair of vertices $u$, $v$ in $G$ there exists an automorphism sending $u$ to $v$. All circulant graphs are well known to be vertex transitive. Thus, by Lemma \ref{l0}, it is enough to find PGST in a circulant graph between the pair of vertices $0$ and $\frac{n}{2}$. First we calculate the $(0,\frac{n}{2})$-th entry of the transition matrix of the cycle $C_n$. Suppose $A=\sum\limits_{l=0}^{n-1}\lambda_lE_l$ is the spectral decomposition of the adjacency matrix of $C_n$. Here notice that $E_l=\frac{1}{n}\vl_l\vl_l^*$, where $\vl_l=\left[1,\omega_n^l,\ldots,\omega_n^{l(n-1)}\right]^T$ as given in $\left(\ref{e}\right)$. Now we obtain the transition matrix of $C_n$ as
\[H(t)=\exp{\left(-itA\right)}=\sum\limits_{l=0}^{n-1}\exp{\left(-i\lambda_l t\right)E_l}.\]
Since $(0,\frac{n}{2})$-th entry of $E_l$ is $\frac{1}{n}\omega_n^{-\frac{nl}{2}}$, we find the $(0,\frac{n}{2})$-th entry of $H(t)$ as 
\begin{eqnarray}\label{e0}
H(t)_{0,\frac{n}{2}}=\frac{1}{n}\sum\limits_{l=0}^{n-1}\exp{\left(-i\lambda_l t\right)}\cdot\omega_{n}^{-\frac{nl}{2}}=\frac{1}{n}\sum\limits_{l=0}^{n-1}\exp{\left[-i\left(\lambda_l t+l\pi\right)\right]}.
\end{eqnarray}
The next result finds some cycles admitting PGST with respect to a sequence in $2\pi\Zl.$ Later we extend this result to more general circulant graphs.
\begin{lem}\label{l1}\cite{pal4}
A cycle $C_{n}$ exhibits pretty good state transfer if $n=2^k,\;k\geq 2$, with respect to a sequence in $2\pi\Zl.$ 
\end{lem}
A graph $G$ is said to be almost periodic if there is a sequence $\left\lbrace t_k\right\rbrace$ of real numbers and a complex number $\gamma$ of unit modulus such that $\lim\limits_{k\rightarrow\infty} H(t_k) =\gamma I,$ where $I$ is the identity matrix of appropriate order. Since the circulant graphs are vertex transitive, we observe that a circulant graph is almost periodic if and only if $\lim\limits_{k\rightarrow\infty} H(t_k) {\bf e}_0=\gamma {\bf e}_0.$ This implies that if a cycle admits PGST then it is necessarily almost periodic. Hence, by Lemma \ref{l1}, all cycles of size $n=2^k,\;k\geq 2$, are almost periodic. However, in the next result, we find a connection between the time sequences of two cycles one of which exhibits PGST and the other one is almost periodic.
\begin{lem}\label{mt1}
Let $k,k'\in \Nl$ and $2\leq k'<k$. If the cycle $C_{2^k}$ admits pretty good state transfer with respect to a sequence $\left\lbrace t_m\right\rbrace$ then $C_{2^{k'}}$ is almost periodic with respect to a sub-sequence of $\left\lbrace t_m\right\rbrace$. 
\end{lem}

\begin{proof}
We begin with the observation that if $k'<k$ then
\[\cos{\left(\frac{2l\pi}{2^{k'}}\right)}=\cos{\left(\frac{2\left(l\cdot 2^{k-k'}\right)\pi}{2^{k}}\right)}.\]
Consequently, if $\theta_l$'s and $\lambda_l$'s are the eigenvalues of $C_{2^{k'}}$ and $C_{2^k}$, respectively, then we necessarily have $\theta_l=\lambda_{l\cdot 2^{k-k'}}$, where $0\leq l\leq 2^{k'}-1.$\par
Suppose $C_{2^k}$ admits PGST with respect to a time sequence $\left\lbrace t_m\right\rbrace$. Therefore, by Equation $\left(\ref{e0}\right)$, there exists a complex number $\gamma$ with $|\gamma|=1$ such that 
\[\lim_{m\rightarrow\infty}\sum\limits_{l=0}^{2^k-1}\exp{\left[-i\left(\lambda_l t_m+l\pi\right)\right]}=2^k\gamma.\]
Since the unit circle is compact, we have a subsequence $\left\lbrace t'_m\right\rbrace$ of $\left\lbrace t_m\right\rbrace$ such that
\[\lim_{m\rightarrow\infty}\exp{[-i\left(\lambda_l t'_m+l\pi\right)]}=\gamma,\text{ for }0\leq l\leq 2^k-1.\]
This further gives
\[\lim_{m\rightarrow\infty}\exp{[-i\left(\theta_l t'_m\right)]}=\gamma,\text{ for }0\leq l\leq 2^{k'}-1.\]
Hence we have
\begin{eqnarray}\label{et1}
\lim_{m\rightarrow\infty}\sum\limits_{l=0}^{2^{k'}-1}\exp{\left[-i\left(\theta_l t'_m\right)\right]}=2^{k'}\gamma.
\end{eqnarray}
Notice that the $00$-th entry of the transition matrix of $C_{2^{k'}}$ is $\frac{1}{2^{k'}}\sum\limits_{l=0}^{2^{k'}-1}\exp{\left[-i\left(\theta_l t\right)\right]},~t\in\Rl.$ Therefore, by Equation $\left(\ref{et1}\right)$, we conclude that $C_{2^{k'}}$ is almost periodic with respect to $\left\lbrace t'_m\right\rbrace$. Hence we have the desired result.
\end{proof}
Now we extend Lemma \ref{l1} to obtain more circulant graphs (apart from the cycles) on $2^k$ vertices exhibiting PGST. Also we find some circulant graphs which are almost periodic. It can well be observed that Theorem $7$ appearing in \cite{pal4} becomes a special case to the following theorem.  

\begin{theorem}\label{mt2}
Let $k\in\Nl$ and $n=2^k$. Also let $Cay\left(\Zl_n,S\right)$ be a non-integral circulant graph. Let $d$ be the least among all the divisors of $n$ so that $S\cap S_n(d)$ is a non-empty proper subset of $S_n(d)$. If $\left|S\cap S_n(d)\right|\equiv 2~(mod~4)$ then $Cay\left(\Zl_n,S\right)$ admits pretty good state transfer with respect to a sequence in $2\pi\Zl$. Moreover, if $\left|S\cap S_n(d)\right|\equiv 0~(mod~4)$ then $Cay\left(\Zl_n,S\right)$ is almost periodic with respect to a sequence in $2\pi\Zl$.
\end{theorem}

\begin{proof}
For each $x\in S\cap S_n(d)$, observe that $Cay\left(\Zl_n,\left\lbrace x,n-x\right\rbrace\right)$ is isomorphic to the disjoint union of $d$ copies of the cycle $Cay\left(\Zl_\frac{n}{d},\left\lbrace1,\frac{n}{d}-1\right\rbrace\right)$. By Lemma \ref{l1}, there exists a sequence $\left\lbrace t_m\right\rbrace$ in $2\pi\Zl$ such that $Cay\left(\Zl_\frac{n}{d},\left\lbrace1,\frac{n}{d}-1\right\rbrace\right)$ admits PGST between the pair of vertices $0$ and $\frac{n}{2d}$ with respect to $\left\lbrace t_m\right\rbrace$. In particular, the cycle $\left(0,x,\ldots,\frac{nx}{2d},\ldots,\left(\frac{n}{d}-1\right)x\right)$ admits PGST between the pair of vertices $0$ and $\frac{nx}{2d}$. Note that $\frac{nx}{2d}\equiv\frac{n}{2}~(mod~n)$. This in turn implies that $Cay\left(\Zl_n,\left\lbrace x,n-x\right\rbrace\right)$ exhibits PGST between the pair of vertices $0$ and $\frac{n}{2}$ with respect to $\left\lbrace t_m\right\rbrace$. If $H_{S\cap S_n(d)}(t)$ is the transition matrix of $Cay\left(\Zl_n,S\cap S_n(d)\right)$ then by Proposition \ref{4ab}, we can express $H_{S\cap S_n(d)}(t)$ as
\begin{eqnarray}\label{1mt2}
H_{S\cap S_n(d)}(t)=\prod\limits_{x\in S\cap S_n(d),~x<\frac{n}{2}}H_{x}(t),
\end{eqnarray}
where $H_{x}(t)$ denotes the transition matrix of $Cay\left(\Zl_n,\left\lbrace x,n-x\right\rbrace\right)$ for all $x\in S\cap S_n(d)$. In Equation $\left(\ref{1mt2}\right)$, notice that if $\left|S\cap S_n(d)\right|\equiv 2~(mod~4)$ then there are odd number of matrices in the product. Also we have just concluded that for all $x\in S\cap S_n(d)$
\begin{eqnarray}\label{11mt2}
\lim_{m\rightarrow\infty} H_{x}(t_m){\bf e}_0=\gamma{\bf e}_{\frac{n}{2}},~|\gamma|=1.
\end{eqnarray}
Since transition matrices are symmetric, we conclude that
\begin{eqnarray}\label{12mt2}
\lim_{m\rightarrow\infty} H_{x}(t_m){\bf e}_{\frac{n}{2}}=\gamma{\bf e}_0,~\forall x\in S\cap S_n(d).
\end{eqnarray}
After successive application of $\left(\ref{11mt2}\right)$ and $\left(\ref{12mt2}\right)$, we obtain the following from Equation $\left(\ref{1mt2}\right)$.
\begin{eqnarray}\label{2mt2}
\lim_{m\rightarrow\infty} H_{S\cap S_n(d)}(t_m){\bf e}_0=\lim_{m\rightarrow\infty} \left[\prod\limits_{x\in S\cap S_n(d),~x<\frac{n}{2}}H_{x}(t_m)\right]{\bf e}_0=\gamma^{\frac{|S\cap S_n(d)|}{2}}{\bf e}_{\frac{n}{2}}.
\end{eqnarray}
Let $S'=\left\lbrace x\in S~|~ gcd(x,n)=d'~\text{and}~S_{d'}\subset S\right\rbrace.$ Notice that if $S'$ is non-empty then it is in fact a gcd-set. Therefore, Theorem \ref{so} implies that the sub-graph $Cay\left(\Zl_n,S'\right)$ of $Cay\left(\Zl_n,S\right)$ is integral. Suppose $H_{S'}(t)$ is the transition matrix of $Cay\left(\Zl_n,S'\right)$. Using spectral decomposition of $H_{S'}(t)$, we deduce that
\begin{eqnarray}\label{21mt2}
H_{S'}(t)=I,~t\in 2\pi\Zl.
\end{eqnarray} 
Now consider $S''=S\setminus\left(S'\cup \left(S\cap S_n(d)\right)\right).$ By the assumption on the divisor $d$ of $n$ and the choice of $S''$, we must have $gcd(s,n)>d$ for all $s\in S''$. This in turn implies that the size of each disjoint cycles appearing in $Cay\left(\Zl_n,\left\lbrace s,n-s\right\rbrace\right)$ is strictly less than the size of the cycle $Cay\left(\Zl_\frac{n}{d},\left\lbrace 1,\frac{n}{d}-1\right\rbrace\right)$. Recall that $Cay\left(\Zl_\frac{n}{d},\left\lbrace 1,\frac{n}{d}-1\right\rbrace\right)$ admits PGST with respect to the sequence $\left\lbrace t_m\right\rbrace$. Therefore, by Theorem \ref{mt1}, we have a subsequence of $\left\lbrace t_m\right\rbrace$ with respect to which each component of $Cay\left(\Zl_n,\left\lbrace s,n-s\right\rbrace\right)$ is almost periodic. Hence $Cay\left(\Zl_n,\left\lbrace s,n-s\right\rbrace\right)$ is also almost periodic. Since $S''$ is finite, we can appropriately choose a subsequence $\left\lbrace t_{m_r}\right\rbrace$ of $\left\lbrace t_m\right\rbrace$ such that for all $s\in S''$, the graph $Cay\left(\Zl_n,\left\lbrace s,n-s\right\rbrace\right)$ is almost periodic with respect to $\left\lbrace t_{m_r}\right\rbrace.$ Hence, by Proposition \ref{4ab}, the graph $Cay\left(\Zl_n,S''\right)$ is almost periodic with respect to $\left\lbrace t_{m_r}\right\rbrace.$ Therefore, if $H_{S''}(t)$ is the transition matrix of $Cay\left(\Zl_n,S''\right)$ then we have
\begin{eqnarray}\label{3mt2}
\lim_{r\rightarrow\infty}H_{S''}\left(t_{m_r}\right){\bf e}_0=\zeta {\bf e}_0,~|\zeta|=1.
\end{eqnarray}
Since $\left\lbrace t_{m_r}\right\rbrace$ is a subsequence of $\left\lbrace t_m\right\rbrace$, notice that $Cay\left(\Zl_n,S\cap S_n(d)\right)$ also admits PGST with respect to $\left\lbrace t_{m_r}\right\rbrace.$ Therefore, from $\left(\ref{2mt2}\right)$, we get
\begin{eqnarray}\label{4mt2}
\lim_{r\rightarrow\infty} H_{S\cap S_n(d)}\left(t_{m_r}\right){\bf e}_0=\gamma^{\frac{|S\cap S_n(d)|}{2}}{\bf e}_{\frac{n}{2}}.
\end{eqnarray}
Note that $S\cap S_n(d),~S'$ and $S''$ are disjoint subsets of $S$, and $S=\left(S\cap S_n(d)\right)\cup S'\cup S''$. Therefore, if $H(t)$ is the transition matrix of $Cay\left(\Zl_n,S\right)$ then by Proposition \ref{4ab}, using $\left(\ref{21mt2}\right),~\left(\ref{3mt2}\right)$ and $\left(\ref{4mt2}\right)$, we obtain
\begin{eqnarray*}
\lim_{r\rightarrow\infty}H\left(t_{m_r}\right){\bf e}_0 &=& \lim_{r\rightarrow\infty}\left[H_{S\cap S_n(d)}\left(t_{m_r}\right)\cdot H_{S''}\left(t_{m_r}\right)\cdot H_{S'}\left(t_{m_r}\right)\right]{\bf e}_0 \\
&=& \lim_{r\rightarrow\infty} H_{S\cap S_n(d)}\left(t_{m_r}\right)\cdot \lim_{r\rightarrow\infty}H_{S''}\left(t_{m_r}\right){\bf e}_0,~\text{as}~t_{m_r}\in 2\pi\Zl\\
&=& \lim_{r\rightarrow\infty} H_{S\cap S_n(d)}\left(t_{m_r}\right)\cdot\zeta{\bf e}_0\\
&=&\zeta\gamma^{\frac{|S\cap S_n(d)|}{2}}{\bf e}_{\frac{n}{2}}.
\end{eqnarray*}
This proves our claim that $Cay\left(\Zl_n,S\right)$ admits PGST with respect to a sequence in $2\pi\Zl$.\\
For the second part of the proof, observe that, if we consider $\left|S\cap S_n(d)\right|\equiv 0~(mod~4)$ then we get the following equation instead of $\left(\ref{2mt2}\right)$.
\begin{eqnarray*}
\lim_{m\rightarrow\infty} H_{S\cap S_n(d)}(t_m){\bf e}_0=\lim_{m\rightarrow\infty} \left[\prod\limits_{x\in S\cap S_n(d),~x<\frac{n}{2}}H_{x}(t_m)\right]{\bf e}_0=\gamma^{\frac{|S\cap S_n(d)|}{2}}{\bf e}_0.
\end{eqnarray*}
Then imitating the rest of the proof of the first part, we obtain the desired result.
\end{proof}

Now we find even more circulant graphs that is either almost periodic or admits PGST. 

\begin{cor}\label{mt3}
Let all the conditions of Theorem \ref{mt2} are satisfied. Let $q\in2\Nl+1$, $N=nq$ and $D$ be any set of proper divisors of $N$ such that $qS\cap S_N(D)=\emptyset$. If $\left|S\cap S_n(d)\right|\equiv 2~(mod~4)$ then $Cay\left(\Zl_N,qS\cup S_N(D)\right)$ admits PGST with respect to a sequence in $2\pi\Zl$. Moreover, if $\left|S\cap S_n(d)\right|\equiv 0~(mod~4)$ then $Cay\left(\Zl_N,qS\cup S_N(D)\right)$ is almost periodic with respect to a sequence in $2\pi\Zl$. 
\end{cor}
\begin{proof}
Note that $S_N(D)$ is a gcd-set of $\Zl_N$. Therefore, by Theorem \ref{so}, the graph $Cay\left(\Zl_N,S_N(D)\right)$ is integral. If $H_{S_N(D)}(t)$ be the transition matrix of $Cay\left(\Zl_N,S_N(D)\right)$ then using spectral decomposition of $H_{S_N(D)}(t)$, we deduce that
\begin{eqnarray}\label{1mt3}
H_{S_N(D)}(t)=I,~t\in 2\pi\Zl.
\end{eqnarray}
Let $H_{qS}(t)$ be the transition matrix of $Cay\left(\Zl_N,qS\right)$. Since $qS\cap S_N(D)=\emptyset$, by Proposition \ref{4ab}, we find the transition matrix of $Cay\left(\Zl_N,qS\cup S_q(D)\right)$ as
\begin{eqnarray*}
H(t)=H_{qS}(t)\cdot H_{S_N(D)}(t),~t\in\Rl
\end{eqnarray*}
If $t\in 2\pi\Zl$ then using $\left(\ref{1mt3}\right)$, we get
\begin{eqnarray}\label{2mt3}
H(t)=H_{qS}(t)
\end{eqnarray}
It can be observed that $Cay\left(\Zl_N,qS\right)$ is isomorphic to $q$ disjoint copies of $Cay\left(\Zl_n,S\right)$. By Theorem \ref{mt2}, if $\left|S\cap S_n(d)\right|\equiv 2~(mod~4)$ then $Cay\left(\Zl_n,S\right)$ admits PGST with respect to a sequence in $2\pi\Zl$. Hence $Cay\left(\Zl_N,qS\right)$ admits PGST with respect to a sequence in $2\pi\Zl$ as well. Similarly, if $\left|S\cap S_n(d)\right|\equiv 0~(mod~4)$ then $Cay\left(\Zl_N,qS\right)$ is almost periodic with respect to a sequence in $2\pi\Zl$. Finally, by $\left(\ref{2mt3}\right)$, we have the desired result.
\end{proof}
Next we show that there are circulant graphs that fails to exhibit PGST. In Lemma $11$ and Corollary $12$ of \cite{pal4}, we observe that if $m\in \Nl$ and $n=mp$, where $p$ is an odd prime, then the cycle $Cay\left(\Zl_n,\left\lbrace1,n-1\right\rbrace\right)$ and its compliment does not exhibit PGST. Now we show that $Cay\left(\Zl_n,S\right)$ does not admit PGST if $p\nmid s$ for all $s\in S$. Hence, Lemma $11$ and Corollary $12$ of \cite{pal4}, becomes a special case of the following Theorem. We use some methods used in \cite{god4} to establish the result.
\begin{theorem}\label{mt4}
Let $m\in \Nl$ and $n=mp$, where $p$ is an odd prime. Then the circulant graph $Cay\left(\Zl_n,S\right)$ does not exhibit pretty good state transfer if $p\nmid s$ for all $s\in S$.  
\end{theorem}

\begin{proof}
If $m$ is an odd number then it is clear from Lemma \ref{l0} that the result follows. Thus, it is enough to prove the result for $m$ even. For an odd prime $p$, let $\omega_p$ be the primitive $p$-th root of unity. Then for $s\in S$, since $p\nmid s$, we have
\[1+\omega_p^s+\omega_p^{2s}+\ldots+\omega_p^{(p-1)s}=0.\]
This further yields
\begin{eqnarray}\label{e2}
1+2\sum\limits_{r=1}^{\frac{p-1}{2}}\cos{\left(\frac{2rs\pi}{p}\right)}=0.
\end{eqnarray}
Multiplying both sides of $\left(\ref{e2}\right)$ by $2\cos{\left(\frac{2s\pi}{n}\right)}$ we obtain the following relation of eigenvalues of $C_n$ (as given in $\left(\ref{e}\right)$).
\begin{eqnarray}\label{e3}
\lambda_s + \sum\limits_{r=1}^{\frac{p-1}{2}}\lambda_{s(mr+1)} + \sum\limits_{r=1}^{\frac{p-1}{2}} \lambda_{s(mr-1)}=0,~s\in S.
\end{eqnarray}
Similarly multiplying $\left(\ref{e2}\right)$ by $2\cos{\left(\frac{4s\pi}{n}\right)}$ gives
\begin{eqnarray}\label{e4}
\lambda_{2s} + \sum\limits_{r=1}^{\frac{p-1}{2}}\lambda_{s(mr+2)} + \sum\limits_{r=1}^{\frac{p-1}{2}} \lambda_{s(mr-2)}=0,~s\in S.
\end{eqnarray}
Note that if $\lambda'_l$ is an eigenvalue of $Cay\left(\Zl_n,S\right)$ then $\lambda'_l=\frac{1}{2}\sum\limits_{s\in S}\lambda_{sl}$. Hence, by Equation $\left(\ref{e3}\right)$ and $\left(\ref{e4}\right)$, we obtain
\begin{eqnarray}\label{e5}
\left(\lambda'_2 - \lambda'_1\right) + \sum\limits_{r=1}^{\frac{p-1}{2}}\left(\lambda'_{mr+2} - \lambda'_{mr+1}\right) + \sum\limits_{r=1}^{\frac{p-1}{2}}\left(\lambda'_{mr-2}-\lambda'_{mr-1}\right)=0.
\end{eqnarray}
If $Cay\left(\Zl_n,S\right)$ admits PGST then, by Equation $\left(\ref{e0}\right)$, we have a sequence of real numbers $\left\lbrace t_k\right\rbrace$ and a complex number $\gamma$ with $|\gamma|=1$ such that 
\[\lim_{k\rightarrow\infty}\sum\limits_{l=0}^{n-1}\exp{\left[-i\left(\lambda'_l t_k+l\pi\right)\right]}=n\gamma.\]
Since the unit circle is compact, we have a subsequence $\left\lbrace t'_k\right\rbrace$ of $\left\lbrace t_k\right\rbrace$ such that
\[\lim_{k\rightarrow\infty}\exp{[-i\left(\lambda'_l t'_k+l\pi\right)}=\gamma,\text{ for }0\leq l\leq n-1.\] 
Moreover, this gives
\[\lim_{k\rightarrow\infty}\exp{[-i\left(\lambda'_{l+1}-\lambda'_l\right)t'_k]}=-1.\]
Denoting the term in left hand side of Equation $\left(\ref{e5}\right)$ as $L$, it is evident that
\[\lim_{k\rightarrow\infty}\exp{\left(-iLt'_k\right)}=-1,\]
but this is absurd as $L=0$. Hence the circulant graph $Cay\left(\Zl_n,S\right)$ does not exhibit PGST if $p\nmid s$ for all $s\in S$.
\end{proof}

\section{Conclusions}
We have considered PGST on circulant graphs. In \cite{pal4}, the authors have classified all cycles and their compliments admitting PGST. There it is established that a cycle $C_n$ exibits PGST if and only if $n$ is a power of $2$. In this article, we have found several more circulant graphs on $n$ vertices exhibiting PGST whenever $n$ is a power of $2$. Observe that, if $n$ has an odd prime factor then there is no cycles on $n$ vertices exhibiting PGST. However, in that case, we have found that there are other circulant graphs admitting PGST. Also we have found a large number of circulant graphs that do not admit PGST. The results extablished in this article generalize some of those results appearing in \cite{pal4}. However, more work is required to classify all circulant graphs exhibiting PGST.

\end{document}